\documentclass{amsart}
\usepackage{latexsym}
\usepackage{amssymb, amsbsy, amsthm, amsmath, amstext, amsopn, verbatim}
\usepackage[color,all]{xy}
\usepackage{amsfonts}
\usepackage{amscd}
\usepackage{mathrsfs}
\usepackage{verbatim}
\usepackage{xcolor}
\usepackage{tabularx}

\input xy
\xyoption{all}

\usepackage{parcolumns}

\newtheorem{thm}{Theorem} [section]
\newtheorem{lemma}[thm]{Lemma}

\newtheorem{corollary}[thm]{Corollary}
\newtheorem{prop}[thm]{Proposition}
\newtheorem{notation}[thm]{Notation}

\newtheorem*{rough-thm-1}{Rough Version of Vanishing Theorem}
\newtheorem*{rough-thm-2}{Rough Version of Exactness Theorem}

\newtheorem*{main example}{Main Example}

\theoremstyle{definition}

\newtheorem*{basic convention}{Basic Conventions}

\newtheorem{defn}[thm]{Definition}

\newtheorem{convention}[thm]{Convention}

\theoremstyle{remark}

\newtheorem{remark}[thm]{Remark}

\newtheorem{fact}[thm]{Fact}

\begin{document}

\numberwithin{equation}{section}

\newcommand{\hs}{\mbox{\hspace{.4em}}}
\newcommand{\ds}{\displaystyle}
\newcommand{\bd}{\begin{displaymath}}
\newcommand{\ed}{\end{displaymath}}
\newcommand{\bcd}{\begin{CD}}
\newcommand{\ecd}{\end{CD}}

\newcommand{\proj}{\operatorname{Proj}}
\newcommand{\bproj}{\underline{\operatorname{Proj}}}
\newcommand{\spec}{\operatorname{Spec}}
\newcommand{\bspec}{\underline{\operatorname{Spec}}}
\newcommand{\pline}{{\mathbf P} ^1}
\newcommand{\pplane}{{\mathbf P}^2}
\newcommand{\coker}{{\operatorname{coker}}}
\newcommand{\ldb}{[[}
\newcommand{\rdb}{]]}

\newcommand{\Sym}{\operatorname{Sym}^{\bullet}}
\newcommand{\Symp}{\operatorname{Sym}}
\newcommand{\Pic}{\operatorname{Pic}}
\newcommand{\AAut}{\operatorname{Aut}}
\newcommand{\PAut}{\operatorname{PAut}}

\newcommand{\too}{\twoheadrightarrow}
\newcommand{\C}{{\mathbf C}}
\newcommand{\cA}{{\mathcal A}}
\newcommand{\cS}{{\mathcal S}}
\newcommand{\cV}{{\mathcal V}}
\newcommand{\cM}{{\mathcal M}}
\newcommand{\bA}{{\mathbf A}}
\newcommand{\aline}{\mathbb{A}^1}
\newcommand{\cB}{{\mathcal B}}
\newcommand{\cC}{{\mathcal C}}
\newcommand{\cD}{{\mathcal D}}
\newcommand{\D}{{\mathcal D}}
\newcommand{\cs}{{\mathbf C} ^*}
\newcommand{\boldc}{{\mathbf C}}
\newcommand{\cE}{{\mathcal E}}
\newcommand{\cF}{{\mathcal F}}
\newcommand{\cG}{{\mathcal G}}
\newcommand{\G}{{\mathbf G}}
\newcommand{\fg}{{\mathfrak g}}
\newcommand{\ft}{\mathfrak t}
\newcommand{\bH}{{\mathbf H}}
\newcommand{\cH}{{\mathcal H}}
\newcommand{\cI}{{\mathcal I}}
\newcommand{\cJ}{{\mathcal J}}
\newcommand{\cK}{{\mathcal K}}
\newcommand{\cL}{{\mathcal L}}
\newcommand{\baL}{{\overline{\mathcal L}}}
\newcommand{\M}{{\mathcal M}}
\newcommand{\bM}{{\mathbf M}}
\newcommand{\bm}{{\mathbf m}}
\newcommand{\cN}{{\mathcal N}}
\newcommand{\theo}{\mathcal{O}}
\newcommand{\cP}{{\mathcal P}}
\newcommand{\cR}{{\mathcal R}}
\newcommand{\boldp}{{\mathbf P}}
\newcommand{\boldq}{{\mathbf Q}}
\newcommand{\bbL}{{\mathbf L}}
\newcommand{\cQ}{{\mathcal Q}}
\newcommand{\cO}{{\mathcal O}}
\newcommand{\Oo}{{\mathcal O}}
\newcommand{\OX}{{\Oo_X}}
\newcommand{\OY}{{\Oo_Y}}
\newcommand{\otY}{{\underset{\OY}{\ot}}}
\newcommand{\otX}{{\underset{\OX}{\ot}}}
\newcommand{\cU}{{\mathcal U}}
\newcommand{\cX}{{\mathcal X}}
\newcommand{\cW}{{\mathcal W}}
\newcommand{\boldz}{{\mathbf Z}}
\newcommand{\cZ}{{\mathcal Z}}
\newcommand{\qgr}{\operatorname{qgr}}
\newcommand{\gr}{\operatorname{gr}}
\newcommand{\coh}{\operatorname{coh}}
\newcommand{\End}{\operatorname{End}}
\newcommand{\Hom}{\operatorname{Hom}}
\newcommand{\uHom}{\underline{\operatorname{Hom}}}
\newcommand{\uHomY}{\uHom_{\OY}}
\newcommand{\uHomX}{\uHom_{\OX}}
\newcommand{\Ext}{\operatorname{Ext}}
\newcommand{\bExt}{\operatorname{\bf{Ext}}}
\newcommand{\Tor}{\operatorname{Tor}}

\newcommand{\inv}{^{-1}}
\newcommand{\airtilde}{\widetilde{\hspace{.5em}}}
\newcommand{\airhat}{\widehat{\hspace{.5em}}}
\newcommand{\nt}{^{\circ}}
\newcommand{\del}{\partial}

\newcommand{\supp}{\operatorname{supp}}
\newcommand{\GK}{\operatorname{GK-dim}}
\newcommand{\hd}{\operatorname{hd}}
\newcommand{\id}{\operatorname{id}}
\newcommand{\res}{\operatorname{res}}
\newcommand{\lrar}{\leadsto}
\newcommand{\im}{\operatorname{Im}}
\newcommand{\HH}{\operatorname{H}}
\newcommand{\TF}{\operatorname{TF}}
\newcommand{\Bun}{\operatorname{Bun}}
\newcommand{\Hilb}{\operatorname{Hilb}}
\newcommand{\Fact}{\operatorname{Fact}}
\newcommand{\F}{\mathcal{F}}
\newcommand{\nthord}{^{(n)}}
\newcommand{\Aut}{\underline{\operatorname{Aut}}}
\newcommand{\Gr}{\operatorname{Gr}}
\newcommand{\Fr}{\operatorname{Fr}}
\newcommand{\GL}{\operatorname{GL}}
\newcommand{\gl}{\mathfrak{gl}}
\newcommand{\SL}{\operatorname{SL}}
\newcommand{\ff}{\footnote}
\newcommand{\ot}{\otimes}
\def\Ext{\operatorname {Ext}}
\def\Hom{\operatorname {Hom}}
\def\Ind{\operatorname {Ind}}
\def\bbZ{{\mathbb Z}}

\newcommand{\nc}{\newcommand}
\newcommand{\on}{\operatorname}
\nc{\cont}{\on{cont}}
\nc{\rmod}{\on{mod}}
\nc{\Mtil}{\widetilde{M}}
\nc{\wb}{\overline}
\nc{\wt}{\widetilde}
\nc{\wh}{\widehat}
\nc{\sm}{\setminus}
\nc{\mc}{\mathcal}
\nc{\mbb}{\mathbb}
\nc{\Mbar}{\wb{M}}
\nc{\Nbar}{\wb{N}}
\nc{\Mhat}{\wh{M}}
\nc{\pihat}{\wh{\pi}}
\nc{\JYX}{\cJ_{Y\leftarrow X}}
\nc{\phitil}{\wt{\phi}}
\nc{\Qbar}{\wb{Q}}
\nc{\DYX}{\D_{Y\leftarrow X}}
\nc{\DXY}{\D_{X\to Y}}
\nc{\dR}{\stackrel{\bbL}{\underset{\D_X}{\ot}}}
\nc{\Winfi}{\cW_{1+\infty}}
\nc{\K}{{\mc K}}
\nc{\unit}{{\bf \on{unit}}}
\nc{\boxt}{\boxtimes}
\nc{\xarr}{\stackrel{\rightarrow}{x}}
\nc{\Cnatbar}{\overline{C}^{\natural}}
\nc{\oJac}{\overline{\on{Jac}}}
\nc{\gm}{{\mathbf G}_m}
\nc{\Loc}{\on{Loc}}
\nc{\Bm}{\operatorname{Bimod}}
\nc{\lie}{{\mathfrak g}}
\nc{\lb}{{\mathfrak b}}
\nc{\lien}{{\mathfrak n}}
\nc{\e}{\epsilon}
\nc{\eu}{\mathsf{eu}}

\nc{\Gm}{{\mathbb G}_m}
\nc{\Gabar}{\wb{\G}_a}
\nc{\Gmbar}{\wb{\G}_m}
\nc{\PD}{{\mathbb P}_{\D}}
\nc{\Pbul}{P_{\bullet}}
\nc{\PDl}{{\mathbb P}_{\D(\lambda)}}
\nc{\PLoc}{\mathsf{MLoc}}
\nc{\Tors}{\on{Tors}}
\nc{\PS}{{\mathsf{PS}}}
\nc{\PB}{{\mathsf{MB}}}
\nc{\Pb}{{\underline{\operatorname{MBun}}}}
\nc{\Ht}{\mathsf{H}}
\nc{\bbH}{\mathbb H}
\nc{\gen}{^\circ}
\nc{\Jac}{\operatorname{Jac}}
\nc{\sP}{\mathsf{P}}
\nc{\sT}{\mathsf{T}}
\nc{\bP}{{\mathbb P}}
\nc{\otc}{^{\otimes c}}
\nc{\Det}{\mathsf{det}}
\nc{\PL}{\on{ML}}
\nc{\sL}{\mathsf{L}}

\nc{\ml}{{\mathcal S}}
\nc{\Xc}{X_{\on{con}}}
\nc{\Z}{{\mathbb Z}}
\nc{\resol}{\mathfrak{X}}
\nc{\map}{\mathsf{f}}
\nc{\gK}{\mathbb{K}}
\nc{\bigvar}{\mathsf{W}}
\nc{\Tmax}{\mathsf{T}^{md}}

\nc{\Cpt}{\mathbb{P}}
\nc{\pv}{e}

\nc{\Qgtr}{Q^{\on{gtr}}}
\nc{\algtr}{\alpha^{\on{gtr}}}
\nc{\Ggtr}{\mathbb{G}^{\on{gtr}}}
\nc{\chigtr}{\chi^{\on{gtr}}}

\newcommand{\la}{\langle}
\newcommand{\ra}{\rangle}
\newcommand{\fm}{\mathfrak m}
\newcommand{\Ms}{\mathfrak M}
\newcommand{\Ma}{\mathfrak M_0}
\newcommand{\ML}{\mathfrak L}
\newcommand{\ev}{\textit{ev}}

\nc{\Higgst}{\mathcal{M}_{\on{Higgs}}^{\on{ss}}(G,C)}
\nc{\Higgsp}{M_{\on{Higgs}}^{\on{ss}}(G,C)}
\nc{\hGm}{\widehat{\mathbb{G}}_m}

\nc{\spin}{\mathsf{K}^{1/2}}
\nc{\MHiggs}{\mathcal{M}_{\on{Higgs}}(G,C)}

\nc{\fixed}[1]{\big(#1\big)_{\on{fixed}}}

\numberwithin{equation}{section}

\title{On the Kirwan Map for Moduli of Higgs Bundles}

\author{Emily Cliff}
\address{Department of Mathematics\\University of Illinois at Urbana-Champaign\\Urbana, IL 61801 USA}
\email{ecliff@illinois.edu}
\author{Thomas Nevins}
\address{Department of Mathematics\\University of Illinois at Urbana-Champaign\\Urbana, IL 61801 USA}
\email{nevins@illinois.edu}
\author{Shiyu Shen}
\address{Department of Mathematics\\University of Illinois at Urbana-Champaign\\Urbana, IL 61801 USA}
\email{sshen16@illinois.edu}


\begin{abstract}
Let $C$ be a smooth complex projective curve and $G$ a connected complex reductive group.  We prove that if the center $Z(G)$ of $G$ is disconnected, then the Kirwan map 
$H^*\big(\Bun(G,C),\mathbb{Q}\big)\rightarrow H^*\big(\Higgst,\mathbb{Q}\big)$ from the cohomology of the moduli stack of $G$-bundles to the moduli stack of semistable $G$-Higgs bundles, fails to be surjective: more precisely, the {\em variant cohomology} (and variant intersection cohomology) of $\Higgst$ is always nontrivial.  We also show that the image of the pullback map $H^*\big(\Higgsp,\mathbb{Q}\big)\rightarrow H^*\big(\Higgst,\mathbb{Q}\big)$, from the cohomology of the moduli space of semistable $G$-Higgs bundles to the stack, cannot be contained in the image of the Kirwan map.  The proof uses  a Borel-Quillen--style localization result for equivariant cohomology of stacks to reduce to an explicit construction and calculation.
\end{abstract}

\maketitle

\section{Introduction}
Let $C$ be a nonsingular complex projective curve of genus $g\geq 2$.  Let $G$ be a connected complex reductive group.  Associated to $C$ and $G$ there are:
\begin{itemize}
\item a moduli stack $\Higgst$ of semistable $G$-Higgs bundles on $C$.
\item a coarse moduli space $\Higgsp$ of semistable $G$-Higgs bundles on $C$ with a morphism $\Higgst\xrightarrow{p}\Higgsp$.
\end{itemize}
The connected components of the moduli stack $\Bun(G,C)$ and the stack $\Higgst$ and space $\Higgsp$ are indexed by the abelian group $\pi_1(G)$ (cf. \cite{GO} for the latter).

Let $Z=Z(G)$ denote the center of $G$.   The group $Z$ can be split as a product, $Z\cong Z^\circ\times\pi_0(Z)$, where the identity component of $Z$ is $Z^\circ$ and the component group is $\pi_0(Z)$.  The abelian group stack $\Bun(Z,C)$, the moduli stack of principal $Z$-bundles, acts on $\Higgst$, inducing an action of the coarse moduli space $\underline{\on{Bun}}(Z,C)$ on $\Higgsp$.  
The fundamental group $\pi_1(G)$ contains a finite-index subgroup, the image of $\iota: \pi_1(Z^\circ)\rightarrow \pi_1(G)$ (cf. \eqref{components subgroup} and Proposition \ref{how components behave}). 
\subsection{Kirwan Map}
For each $\eta\in\pi_1(G)$  write $\Bun(G,C)_\eta$, $\Higgst_\eta$, $\Higgsp_\eta$ for the corresponding components.  
Applying singular cohomology with $\mathbb{Q}$-coefficients, the maps 
\bd
\xymatrix{\Bun(G,C)_\eta & \Higgst_\eta \ar[l]_{k_\eta} \ar[r]^{p_\eta} & \Higgsp_\eta} \hspace{2em} \text{induce a diagram}
\ed
\bd
\xymatrix{
H^*\big(\Bun(G,C)_\eta\big) \ar[r]^{\kappa_\eta = k_\eta^*} & H^*\big(\Higgst_\eta\big) & H^*\big(\Higgsp_\eta\big). \ar[l]_{p_\eta^*} 
}
\ed
The left-hand map $\kappa_\eta$ is the {\em Kirwan map} for the moduli stack of semistable $G$-Higgs bundles.  The cohomology $H^*(\Bun(G,C))$ is known to be a free graded-commutative algebra (i.e., a tensor product of a symmetric algebra on generators in even degrees and an exterior algebra on generators in odd degrees).  Only when $G=SL_2, GL_2, PGL_2$ is it fully known (i.e. for all $\eta$)  whether $\kappa_\eta$ is surjective.

\subsection{Main Theorem}
As we explain in Section \ref{sec:Z-Bundles}, the action of $\Bun(Z,C)$ on $\Higgst$ induces compatible actions of the finite group\footnote{More precisely, the moduli stack $\Bun(\pi_0(Z), C)$ is a finite gerbe over a finite group, but the action on cohomology factors through the finite group.}
$\Bun(\pi_0(Z), C)$ on $H^*(\Higgst)$ and $H^*(\Higgsp)$.  The 
$\Bun(\pi_0(Z), C)$-fixed part of $H^*(\Higgst)$ and $H^*(\Higgsp)$ is called the {\em invariant cohomology}.
The quotients by the 
$\Bun(\pi_0(Z), C)$-invariant cohomology
\bd
H^*(\Higgst)^{\on{variant}} := H^*(\Higgst)/H^*(\Higgst)^{\Bun(\pi_0(Z), C)},
\ed
\bd 
H^*(\Higgsp)^{\on{variant}} := H^*(\Higgsp)/H^*(\Higgsp)^{\Bun(\pi_0(Z), C)}
\ed
\noindent
are the {\em variant cohomology}.  One analogously defines {\em variant intersection cohomology} groups \newline $IH^*(\Higgst)^{\on{variant}}$, $IH^*(\Higgsp)^{\on{variant}}$ (with $\mathbb{Q}$-coefficients).

In Definition \ref{specific bundle def}, we identify a particular $G$-bundle, denoted $P_0$; write $\Higgst_{\eta_0}$ for the component of $\Higgst$, indexed by the element $\eta_0\in\pi_1(G)$, in which $P_0$ lies.  We then prove:
\begin{thm}\label{main thm}
If the component group $\pi_0(Z(G))$ is nontrivial, then for every $\eta \in\eta_0+\pi_1(Z^\circ)$:
\begin{enumerate}
\item The variant cohomology groups $H^*\big(\Higgst_\eta\big)^{\on{variant}}$, $H^*\big(\Higgsp_\eta\big)^{\on{variant}}$ 
 contain the quotient 
$\mathbb{Q}[\Bun(\pi_0(Z),C)]/\mathbb{Q}$
of the regular representation of $\Bun(\pi_0(Z),C)$.
\item Similarly, the variant intersection cohomology groups 
$IH^*\big(\Higgst_\eta\big)^{\on{variant}}$, \newline $IH^*\big(\Higgsp_\eta\big)^{\on{variant}}$  contain the quotient 
$\mathbb{Q}[\Bun(\pi_0(Z),C)]/\mathbb{Q}$
of the regular representation of $\Bun(\pi_0(Z),C)$.
\end{enumerate}
In particular,  
\begin{enumerate}
\item[(3)] $\kappa_\eta$ fails to be surjective.
\item[(4)] $p^*_\eta H^*\big(\Higgsp_\eta\big)\not\subseteq \kappa_\eta^* H^*\big(\Bun(G,C)_\eta\big)$.  
\end{enumerate}
\end{thm}
This shows that, for semisimple $G$, the Kirwan map can {\em only} be surjective when $G$ is of adjoint type, generalizing the work of Hitchin (cf. \cite{Hitchin} and \cite{HT}) for $SL_n$.  
\begin{remark}
The proof shows (see Remark \ref{MHS remark 2}) that the representation $\mathbb{Q}[\Bun(\pi_0(Z),C)]$ of 
$\Bun(\pi_0(Z),C)$
 in part (1) of Theorem \ref{main thm} appears in the pure parts 
 \bd
 \bigoplus_k \on{Gr}_k^W\big(H^k\big(\Higgst_\eta\big)\big) \; \text{and}  \; \bigoplus_k \on{Gr}_k^W\big(H^*\big(\Higgsp_\eta\big)\big)
 \ed
  with respect to the mixed Hodge structure.  In particular, the Kirwan map (whose domain $H^*\big(\Bun(G,C)_\eta\big)$ is pure) fails to surject onto the pure part of the cohomology.
\end{remark}

When $Z(G)$ is connected, one {\em does} expect $\kappa$ to be an isomorphism; this is proven when $G=GL_n$ for coprime rank $n$ and degree $d$ (when the map $p$ becomes a $\Gm$-gerbe over a smooth variety) in \cite{Markman}, and for $G=GL_2, PGL_2$ and $d=0$ in \cite{DWWW}, but the bulk of cases seem to remain open.

\subsection{Strategy of Proof}
The proof of Theorem \ref{main thm} makes essential use of the commuting actions of the multiplicative group $\Gm$, acting by scaling the Higgs field, and the abelian group stack $\Bun(Z,C)$ of principal $Z$-bundles, acting by ``tensoring'' the $G$-Higgs bundle, on $\Higgst$ and $\Higgsp$.  Section \ref{sec:Z-Bundles} develops basic properties of the action of $\Bun(Z,C)$, in particular showing that $\Bun(Z,C)$ acts infinitesimally freely on $\Higgst$ and that the induced action of $H^1(C,\pi_0(Z))$, the group of isomorphism classes of principal $\pi_0(Z)$-bundles, preserves connected components of $\Higgst$.  

We show that whenever $H^1(C,\pi_0(Z))$ is nontrivial it acts nontrivially on 
$H^*\big(\Higgst_\eta\big)$  and $H^*\big(\Higgsp_\eta\big)$, whereas it acts trivially on $H^*\big(\Bun(G,C)_\eta\big)$.  Essentially the same idea  was used in \cite{DWWW} (without using the $\Gm$-action to simplify) and \cite{HT}.  

Section \ref{Eq Coh and Localization} develops the principal computational tool in our proof, a Borel-Quillen--style localization theorem for equivariant cohomology of stacks that, applied to the $\Gm$-action, allows us to reduce to a calculation on a certain $\Gm$-fixed locus, $\Higgst_{\on{fixed}}$ in the notation of that section.  We show by explicit construction that $\Higgst_{\on{fixed}}$ contains a union of connected components isomorphic to $\Bun(Z,C)$, compatibly with the action of $\Bun(Z,C)$ on $\Higgst$.  It follows that $H^1(C,\pi_0(Z))$ acts nontrivially on $H^*\big(\Higgst\big)$, whereas it acts trivially on $H^*\big(\Bun(G,C)\big)$.  This proves the first assertion of Theorem \ref{main thm} for the cohomology of the stack (and essentially the same proof gives the intersection cohomology statement); the proof of the assertions for the moduli space is similar.
  
 Although the use of $\mathbb{C}^*$ localization to establish topological features of algebraic varieties or other topological spaces is completely standard, it seems to be new to use it for {\em Artin stacks} rather than merely Deligne-Mumford stacks.  While the statement we use is not hard, it appears to be new.   It turns out that there is broader story of such localizations that seems not to be as well known as we think warranted, and that we plan to develop further in a future paper.

A result related to Theorem \ref{main thm} will appear in \cite{McNkirwan}, providing explicit examples of finite-dimensional hyperk\"ahler manifolds $\mathsf{M}$ with the action of compact groups $K$ for which the hyperk\"ahler Kirwan map $H^*_K(\mathsf{M})\rightarrow H^*(\mathsf{M}/\!\!/\!\!/ K)$  also fails to be surjective.  
 We hope that, with the growing lists of examples in which the hyperk\"ahler Kirwan map is \cite{JKK, Markman, McN-HKK, DWWW} or is not (\cite{DWWW, HT, McNkirwan} and the present paper) surjective, we are developing enough data to suggest a characterization of when $\kappa$ should be surjective.

We also note that we have chosen to include some basic facts that could be cited in the literature rather than developed from scratch as we have done here---while this runs the risk of giving the impression that we claim originality for facts well known to the community, we simply note that even the present more-or-less self-contained paper is not long.

\subsection{Acknowledgments}
We are grateful to Kevin McGerty for many conversations, and to Brian Collier, Kevin McGerty, Steve Rayan, Charles Rezk, Laura Schaposnik,  and Eric Vasserot for help with references.  All three authors were supported by NSF grant DMS-1502125.
The second author was supported by a Simons Foundation fellowship and NSF grant DMS-1802094.

\section{Notation, Conventions, and Basics of $G$-Higgs Bundles}
\subsection{Notation and Conventions}
Throughout the paper we work with varieties, schemes, and stacks over $\mathbb{C}$.  We write $\Gm$ for the multiplicative group $\mathbb{C}^*$.
We write $\mathbb{A}^1(d)$ for $\mathbb{A}^1$ with the $\Gm$-action of weight $d$, that is, $t\cdot v = t^dv$.
\begin{center}
{\em As in the introduction, all singular cohomology groups are taken with $\mathbb{Q}$-coefficients.}
\end{center}

Throughout the paper we fix a smooth projective (connected) complex curve $C$ of genus $g\geq 2$.
Choose a spin bundle $K^{1/2}_C$ and let $\spin$ denote the associated $\Gm$-bundle, so that 
$K^{1/2}_C = \spin \times_{\Gm} \mathbb{A}^1(1)$.  Let $\mathsf{K}^{-1/2}$  denote the dual $\Gm$-bundle, so $\mathsf{K}^{-1/2}\times_{\Gm} \mathbb{A}^1(-1)\cong K^{1/2}_C$.

We also work with a fixed connected, reductive complex algebraic group $G$ with center $Z=Z(G)$. 
We fix a Borel subgroup $B\subset G$ and a maximal torus $T\subset B$; let $U\subset B$ denote the unipotent radical.  We write $\mathfrak{g} = \on{Lie}(G)$, $\mathfrak{b} = \on{Lie}(B)$, $\mathfrak{n} = \on{Lie}(U)$, $\mathfrak{z}  = \on{Lie}(Z)$, and $\mathcal{N}$ for the nilpotent cone of $\mathfrak{g}$.  
We let $\mathfrak{g}^{\on{reg}}$ denote the regular locus of $\mathfrak{g}$.

For an affine algebraic group $H$, we write $\Bun(H,C)$ for the moduli stack of principal $H$-bundles on $C$.  When the identity component of $H$ is an algebraic torus, we write $\underline{\Bun}(H,C)$ for the coarse moduli scheme of principal $H$-bundles.

We use the terminology ``abelian group'' and ``commutative group'' interchangeably.  However, we prefer the phrases ``commutative group scheme/stack'' because of the risk that ``abelian group scheme'' suggests that the scheme is proper (hence an abelian variety) over some given base.

\subsection{Basics on $\Higgst$ and $\Higgsp$}\label{sec:Higgs basics}

Let $(\cE,\theta)$ be a $G$-Higgs bundle on $C$.
Recall that for a parabolic subgroup $\mathsf{P}\subset G$ and $\mathsf{P}$-reduction $\cE_{\mathsf{P}}$ of $\cE$, we say $\cE_{\mathsf{P}}$ is {\em compatible with $\theta$} if $\theta$ is (the image of) a section of $\cE_{\mathsf{P}} \times_{\mathsf{P}} \mathfrak{p}\otimes \Omega^1_C\subset \cE_{\mathsf{P}}\times_{\mathsf{P}}\mathfrak{g}\otimes\Omega^1_C$. 
Then the $G$-Higgs bundle $(\cE,\theta)$ is {\em semistable}, respectively {\em stable}, if for every proper parabolic subgroup ${\mathsf{P}}\subset G$ and ${\mathsf{P}}$-reduction $\cE_{\mathsf{P}}$ of $\cE$ compatible with $\theta$, the bundle $\cE_{\mathsf{P}}\times_{\mathsf{P}} \mathfrak{g}/\mathfrak{p}$ has nonpositive, respectively negative, degree.  The center $Z(G)$ always acts by automorphisms of the pair $(\cE,\theta)$; we say $(\cE,\theta)$ is {\em simple} if $\on{Aut}(\cE,\theta) = Z(G)$.

\begin{notation}
Let $\Higgst$ denote the moduli stack of semistable $G$-Higgs bundles ({\em not} the moduli stack of all $G$-Higgs bundles!) and $\Higgsp$ the coarse moduli space of semistable $G$-Higgs bundles.
\end{notation}

The quasi-projective scheme $\Higgsp$ is a coarse space of $\Higgst$.  
We note some well-known facts (cf. the ``Main Properties'' of \cite{Alper}).
\begin{prop}\label{Higgs basics}
\mbox{}
Consider the morphism $p: \Higgst\rightarrow\Higgsp$.
\begin{enumerate}
\item The morphism $p: \Higgst\rightarrow\Higgsp$ is surjective.
\item Each fiber of $p$ contains a unique closed point of the moduli stack.  
\item The morphism $p$ is universally closed; in particular, the topology of $\Higgsp$ is the quotient topology of the stack $\Higgst$.  
\end{enumerate}
\end{prop}
The conditions of stability, respectively simplicity, define open substacks $\mathcal{M}_{\on{Higgs}}^{\on{s}}(G,C)$, respectively $\mathcal{M}_{\on{Higgs}}^{\on{spl}}(G,C)$ of $\Higgst$.  Write 
$\mathcal{M}_{\on{Higgs}}^{\on{st-spl}}(G,C) := \mathcal{M}_{\on{Higgs}}^{\on{s}}(G,C)\cap \mathcal{M}_{\on{Higgs}}^{\on{spl}}(G,C)$.

For the morphism $p$ of Proposition \ref{Higgs basics}, we have 
\bd
p\inv\big(p(\mathcal{M}_{\on{Higgs}}^{\on{s}}(G,C))\big) = \mathcal{M}_{\on{Higgs}}^{\on{s}}(G,C), \hspace{1em} 
p\inv\big(p(\mathcal{M}_{\on{Higgs}}^{\on{st-spl}}(G,C))\big) = \mathcal{M}_{\on{Higgs}}^{\on{st-spl}}(G,C).
\ed
\begin{notation}
We write:
$M^{\on{s}}_{\on{Higgs}}(G,C) := p\big(\mathcal{M}_{\on{Higgs}}^{\on{s}}(G,C)\big)$ and
$M^{\on{st-spl}}_{\on{Higgs}}(G,C) := p\big(\mathcal{M}_{\on{Higgs}}^{\on{st-spl}}(G,C)\big)$.
\end{notation}
\begin{lemma}\label{BZ torsor}
The morphism $\mathcal{M}_{\on{Higgs}}^{\on{st-spl}}(G,C)\rightarrow M^{\on{st-spl}}_{\on{Higgs}}(G,C)$ is a torsor for the commutative group stack $BZ$, inducing an identification
$M^{\on{st-spl}}_{\on{Higgs}}(G,C) = \mathcal{M}_{\on{Higgs}}^{\on{st-spl}}(G,C)/BZ$.
\end{lemma}

Consider the product $C\times\Bun(G,C)\rightarrow BG\times B\Gm$ of the classifying morphism
$C\times\Bun(G,C)\rightarrow BG$ for the universal $G$-bundle and the composite 
$C\times\Bun(G,C)\rightarrow C \rightarrow B\Gm$ where the latter map classifies the $\Gm$-bundle associated to $\Omega^1_C$.  The universal $G$-Higgs field then determines a morphism 
$C\times\Higgst\rightarrow \mathfrak{g}/(G\times\Gm)$ making the following square commute:
\bd
\xymatrix{C\times\Higgst \ar[r]\ar[d] & \mathfrak{g}/(G\times\Gm) \ar[d]\\
C\times\Bun(G,C)\ar[r] & BG\times B\Gm.
}
\ed
If $S\subset\mathfrak{g}$ is $G\times\Gm$-stable, we slightly abuse terminology by saying $(P,\theta)$ {\em takes values in $S/G\times\Gm$} if the restricted universal morphism $C\times\{(P,\theta)\}\rightarrow \mathfrak{g}/G\times\Gm$ does.  
A $G$-Higgs pair $(P, \theta)$ is {\em regular} if it  takes values in $\mathfrak{g}^{\on{reg}}/G\times\Gm$ and {\em nilpotent} if it takes values in $\mathcal{N}/G\times\Gm$.
It is immediate from the properness of $C$ that:
\begin{lemma}\label{regular locus}
 The regular $G$-Higgs bundles form an open substack $\Higgst_{\on{reg}}$ of $\Higgst$. 
 \end{lemma} 

\section{Principal $Z(G)$-Bundles and the Action on $G$-Higgs Bundles}\label{sec:Z-Bundles}
\subsection{Principal $Z(G)$-Bundles}
Write $Z=Z(G)$.  We have a short exact sequence
\begin{equation}\label{ex seq of center}
1\rightarrow Z^\circ \rightarrow Z \rightarrow \pi_0(Z)\rightarrow 1,
\end{equation}
where $Z^\circ$ is the identity component.  Note that $Z^\circ$ is a torus and $\pi_0(Z)$ is a finite abelian group.   
\begin{lemma}\label{ab-gp-splitting}
\mbox{}
\begin{enumerate}
\item Suppose
$1\longrightarrow D \longrightarrow A \longrightarrow F\longrightarrow 1$
is an exact sequence of abelian groups, where $F$ is finitely generated abelian and $D$ is a divisible group.  Then $A\cong F\times D$.
\item In particular, for the center $Z(G)$ of a connected reductive group $G$ over $\mathbb{C}$ we have $Z(G)\cong Z^\circ\times\pi_0(Z)$.
\end{enumerate}
\end{lemma}
\begin{remark}\label{choice of splitting}
In light of Lemma \ref{ab-gp-splitting}(2), we henceforth fix a choice of splitting $Z(G)\cong Z^\circ\times\pi_0(Z)$.
\end{remark}

\noindent
For a commutative group scheme $A$ whose identity component is an algebraic torus, write $\underline{\on{Bun}}(A, C) = H^1(C,A)$ for the coarse moduli space of $A$-bundles on $C$.
The splitting $Z(G)\cong Z^\circ\times\pi_0(Z)$ determines an isomorphism  of group schemes
\begin{equation}\label{ses of groups}
\underline{\on{Bun}}(Z,C) = \underline{\on{Bun}}(Z^\circ, C) \times
\underline{\on{Bun}}(\pi_0(Z), C).
\end{equation}

 We have
$\Bun(A,C) \cong B(A)\times \underline{\on{Bun}}(A,C)$ when $A$ is any of the groups $Z^\circ, Z$, or $\pi_0(Z)$; in particular, $\Bun(\pi_0(Z),C)$ is a gerbe over a finite abelian group scheme.
  Since $Z^\circ$ is a torus,  
\begin{equation}
\underline{\on{Bun}}(Z^\circ,C) 
\cong \on{Jac}(C)^r\times \mathbb{Z}^r.
\end{equation}
\begin{prop}\label{properness of BunZ}
We have
\bd
\Bun(Z,C)\cong B(Z^\circ)\times B(\pi_0(Z)) \times \on{Jac}(C)^r \times \pi_0\left(\underline{\Bun}(Z^\circ,C)\right) \times \pi_0\left(\underline{\Bun}(\pi_0(Z),C)\right).
\ed
In particular, $\Bun(Z,C)$ is a disjoint union of connected components, each of which is a $BZ$-torsor over a nonsingular projective variety.
\end{prop}

\subsection{Action of $\Bun(Z,C)$ on $\Bun(G,C)$ and $T^*\Bun(G,C)$}
The stack $\Bun(Z,C)$ of principal $Z$-bundles on $C$ is a commutative group stack, with operation 
$(\cE,\cF)\mapsto \cE \times_Z \cF$.  The group stack $\Bun(Z,C)$ acts naturally on $\Bun(H,C)$ as well by the same formula when $\cF$ is a principal $H$-bundle for any closed subgroup $Z\subseteq H\subseteq G$.  

In particular, the action of $\Bun(Z,C)$ on $\Bun(G,C)$ induces an action of $\Bun(Z,C)$ on $T^*\Bun(G,C)$, the moduli stack of all $G$-Higgs bundles on $C$.  
Because the action of $\Bun(Z,C)$ preserves parabolic reductions of a given $G$-bundle and $Z$ acts trivially on $\mathfrak{g}/\mathfrak{p}$ for any parabolic subalgebra $\mathfrak{p}\subset \mathfrak{g}$:
\begin{remark}\label{preservation of open subsets}
The action of $\Bun(Z,C)$ on $T^*\Bun(G,C)$ preserves the open substacks of semistable, stable, and simple $G$-Higgs bundles (definitions as in Section \ref{sec:Higgs basics}). 
\end{remark} 

Let $(P,\theta)$ be a $G$-Higgs bundle.  One gets a complex of vector bundles over $C$,
\bd
\mathsf{Def}(P,\theta) = P\times_G\mathfrak{g}\xrightarrow{[\theta,-]} P\times_G\mathfrak{g}\otimes \Omega^1_C
\ed
with $P\times_G\mathfrak{g}$ in cohomological degree $-1$.  For the fiber of the tangent complex at $(P,\theta)$ one has 
\bd
\mathbb{T}_{(P,\theta)}\big(T^*\Bun(G,C)\big) \simeq \check{C}^\bullet\big(C,\mathsf{Def}(P,\theta)\big),
\ed
where $\check{C}^\bullet$ denotes the \v{C}ech complex associated to any choice of affine open cover of $C$.

Note that the Lie algebra $\mathfrak{g}$ splits as a $G$-representation, 
\begin{equation}\label{g-splitting}
\mathfrak{g} \cong \mathfrak{z}\oplus \mathfrak{g}_{\on{ad}}.
\end{equation}
It follows that $P\times_G\mathfrak{z}[1] = \theo_C\otimes\mathfrak{z}[1]$ is a direct summand of 
$\mathsf{Def}(P,\theta)$.  Writing $\mathsf{triv}$ for the trivial $Z$-bundle on $C$, we have $\mathbb{T}_{\mathsf{triv}}\big(\Bun(Z,C)\big) \simeq \check{C}^\bullet\big(C,\theo_C\otimes\mathfrak{z}[1]\big)$.  We conclude:
\begin{lemma}
For any choice of $G$-Higgs bundle $(P,\theta)$, the infinitesimal actions of $\Bun(Z,C)$ on $T^*\Bun(G,C)$ and $\Bun(G,C)$,
\bd
\mathbb{T}_{\mathsf{triv}}\big(\Bun(Z,C)\big) \rightarrow \mathbb{T}_{(P,\theta)}\big(T^*\Bun(G,C)\big) 
\rightarrow \mathbb{T}_P\big(\Bun(G,C)\big)
\ed
express $\mathbb{T}_{\mathsf{triv}}\big(\Bun(Z,C)\big)$ as a direct summand of both $\mathbb{T}_{(P,\theta)}\big(T^*\Bun(G,C)\big)$ and $\mathbb{T}_P\big(\Bun(G,C)\big)$.
\end{lemma}
Since $\Bun(Z,C)$ is a gerbe over a smooth scheme, we conclude:
\begin{corollary}\label{injectivity}
For any choice of $G$-Higgs bundle $(P,\theta)$, the induced morphisms of ``action on $P$, respectively $(P,\theta)$,''
\bd
\xymatrix{
\Bun(G,C) & \Bun(Z,C)\ar[l]\ar[r] & T^*\Bun(G,C)
}
\ed
are representable and injective on tangent spaces.
\end{corollary}
\begin{proof}
The homomorphisms $Z(G)\rightarrow \on{Aut}(P,\theta)$ and $Z(G)\rightarrow\on{Aut}(P)$ are always injective, hence the morphisms are representable.  By the lemma, the induced map on $H^0(\mathbb{T})$ is injective, hence the morphisms are immersions.
\end{proof}

\subsection{Connected Components of $\Bun(G,C)$}
Let $\wt{G}$ denote the universal cover of $G$ as a complex-analytic Lie group, with the exact sequence
\begin{equation}\label{universal cover of G}
1\longrightarrow \pi_1(G)\longrightarrow \wt{G}\longrightarrow G \longrightarrow 1.
\end{equation}
Each of these groups determines a complex-analytic sheaf of groups over the complex-analytic curve $C_{\on{an}}$, abusively denoted by the same symbol.  
The set of $\mathbb{C}$-points of the stack $\Bun_G(C)$ is the (non-abelian) cohomology set $H^1(C_{\on{an}},G)$.
We recall (for example, from \cite{Hoffmann}) that the set of connected components of the stack $\Bun_G(C)$ is classified by the connecting homomorphism of cohomology sets
\bd
H^1(C_{\on{an}},G)\xrightarrow{\delta} H^2(C,\pi_1(G)) = \pi_1(G):
\ed
that is, the connected components of $\Bun_G(C)$ are labelled by $\pi_1(G)$, and the fibers of $\delta$ are the isomorphism classes lying in the given connected component.

Via the splitting $Z = Z^\circ\times \pi_0(Z)$ (Remark \ref{choice of splitting}), we get  a homomorphism $\pi_0(Z)\hookrightarrow G$.  Pulling back the exact sequence \eqref{universal cover of G} gives an extension 
\begin{equation}\label{ex seq for pi0Z}
1\longrightarrow \pi_1(G)\longrightarrow \wt{\pi_0(Z)} \longrightarrow \pi_0(Z) \longrightarrow 1
\end{equation}
inducing a commutative square
\bd
\xymatrix{H^1(C,\pi_0(Z)) \ar[r]\ar[d] & H^2(C,\pi_1(G))\ar[d]^{=}\\
H^1(C,G) \ar[r]^{\delta} & H^2(C,\pi_1(G))}
\ed
and making $H^1(C,G)\xrightarrow{\delta}H^2(C,\pi_1(G))$ an equivariant map of $H^1(C,\pi_0(Z))$-sets.
\begin{prop}\label{action on connected components}
The group $H^1(C,\pi_0(Z))$ acts trivially on $H^2(C,\pi_1(G))$.  Thus, the abelian group stack $\Bun(\pi_0(Z),C)$ acts on $\Bun(G,C)$ preserving connected components.
\end{prop}
\begin{proof}
Via the commutative square, we see that it suffices to show that the boundary map $H^1(C,\pi_0(Z))\rightarrow H^2(C,\pi_1(G))$ associated to \eqref{ex seq for pi0Z} is zero.  But since the groups in \eqref{ex seq for pi0Z} are discrete and $H^*(C,\mathbb{Z})$ is a free $\mathbb{Z}$-module, the Universal Coefficient Theorem gives
\bd
H^1(C,\wt{\pi_0(Z)}) = H^1(C,\mathbb{Z})\otimes \wt{\pi_0(Z)} \twoheadrightarrow H^1(C,\mathbb{Z})\otimes \pi_0(Z) = H^1(C,\pi_0(Z)),
\ed
implying that the boundary map is zero.
\end{proof}

\subsection{Action of $\Bun(Z,C)$ on Components of $\Bun(G,C)$}
Now, choose a principal $G$-bundle $P$ on $C$, determining a $\mathbb{C}$-point of $\Bun(G,C)$.  We obtain a morphism
\bd
\psi_P: \Bun(Z,C)\rightarrow \Bun(G,C), \hspace{1em} P_Z \mapsto P_Z\times_Z P.
\ed
\begin{prop}\label{how components behave}
The map of sets of connected components induced by the morphism $\psi_P$ maps $\pi_0\left(\Bun(Z^\circ,C)\right) = \pi_1(Z^\circ)$ injectively to $\pi_0\left(\Bun(G,C)\right) = \pi_1(G)$.  The image  of $\pi_0\left(\Bun(Z^\circ, C)\right) \rightarrow \pi_0\left(\Bun(G,C)\right)$ is a finite-index subgroup.
\end{prop}
\begin{proof}
It suffices to show that the homomorphism $H^2(\iota): H^2(C,\pi_1(Z^\circ))\rightarrow H^2(C,\pi_1(G))$ is injective with finite cokernel.  But up to an isogeny, $G$ is the product of a torus $\mathbb{T}$ and a semisimple group, with the torus $\mathbb{T}$ thus being isogenous to $Z^\circ$.  Since a semisimple group has finite $\pi_1$, we conclude:
\begin{equation}\label{components subgroup}
\text{The map} \hspace{1em} \iota: \pi_1(Z^\circ) \rightarrow \pi_1(G) \hspace{1em} \text{has finite kernel and cokernel.}
\end{equation}
 Since the domain of $\iota$ is torsion-free, the kernel is trivial.  Finally, it follows from the Universal Coefficient Theorem that $H^2(\iota)$ is obtained by tensoring $\pi_1(Z^\circ) \rightarrow \pi_1(G)$ with $H^2(C,\mathbb{Z}) = \mathbb{Z}$.
\end{proof}

\section{Equivariant Cohomology and Partial Localization}\label{Eq Coh and Localization}
In this section, let $V$ be a reduced, separated scheme of finite type over $\mathbb{C}$, viewed as a topological space in the analytic topology, with the action of a complex reductive algebraic group $\Gm\times G$.

\subsection{Fixed Points on Stacks}
The stack $V/G$ admits an action of the multiplicative group $\Gm$ induced from the product action on $V$; we write $(t,v)\mapsto t\cdot v$ for this action.  Replacing that action by $a_m(t,v) = t^m \cdot v$ for each positive integer $m$,  we obtain the action on $V/G$ of the pro-object $\ds\hGm = \lim_{\longleftarrow} \Gm$ in the category of algebraic groups, defined by all coverings of $\Gm$ by itself.  

For simplicity write $\cX = V/G$.  When we wish to consider $\cX$ with its $\Gm$-action via $a_m$ we write $\cX_{a_m}$.  
Recall \cite[Definition~2.3]{Romagny} that a {\em stack of fixed points} $\cX^{\Gm}_{a_m}$ is a stack $F$ that represents the $2$-functor $N\mapsto \Hom_{\Gm-\on{Stacks}}(\iota(N),\cX_{a_m})$ where $\iota$ takes a stack $N$ to the $\Gm$-stack with trivial $\Gm$-action.  
A fixed point $\on{Spec}(\mathbb{C}) \rightarrow \cX^{\Gm}_{a_m}$ can be viewed
as a choice of section of $(\cX_{a_m})/\Gm \rightarrow B\Gm$.

There are morphisms 
$\cX^{\Gm}_{a_m}\rightarrow \cX^{\Gm}_{a_{m\cdot k}}\rightarrow \cX$
 for all integers $m,k>0$, thus defining 
 \bd
\cX^{\hGm} = \lim_{\longrightarrow} \cX^{\Gm}_{a_m} \longrightarrow \cX.
 \ed

For each $x\in V$, write $\on{Lie}(\Gm\times G)_x$ for the Lie algebra of the stabilizer $(\Gm\times G)_x$.  We write 
\bd
\fixed{V/(\Gm\times G)} = \{x\in V \; |\; \on{Lie}(\Gm\times G)_x\rightarrow \on{Lie}(\Gm) \; \text{is surjective}\}/(\Gm\times G).
\ed

\begin{lemma}
The condition above defines a closed substack $\fixed{V/(\Gm\times G)}$ of $V/(\Gm\times G)$.
\end{lemma}
We write $\mathcal{X}_{\on{fixed}}=\fixed{V/G}$ for the pre-image in $V/G$ of $\fixed{V/(\Gm\times G)}$ (cf. Appendix C  of \cite{GP} in the case of DM stacks).
We have:
\begin{prop}\label{fixed points and localization}
The set of images of $\mathbb{C}$-points of $(V/G)^{\hGm}\rightarrow V/G$ is exactly $\fixed{V/G}$.
\end{prop}
\begin{proof}
Suppose first that $x\in V$ and $\overline{x} \in (V/G)^{\hGm}$ have isomorphic images in $V/G$.  Then there exists an $m>0$ such that $\overline{x}$ determines a section of $(V/G)_{a_m}/\Gm \rightarrow B\Gm$. The section induces homomorphisms $\Gm\rightarrow (\Gm\times G)_x\rightarrow \Gm$, where $(\Gm\times G)_x$ is defined with respect to the action of $\Gm$ via $a_m$, whose composite is the identity.

Conversely, suppose $x\in V$ has image in $\fixed{V/G}$, so that $\on{Lie}(\Gm\times G)_x\rightarrow \on{Lie}(\Gm)$ is surjective.  The stabilizer $(\Gm\times G)_x$ is the semi-direct product of its unipotent radical and a Levi subgroup $L_x$, and the unipotent radical maps trivially to $\Gm$; it follows that the Levi subgroup $L_x$ of $(\Gm\times G)_x$ surjects onto $\Gm$.  Since $L_x\rightarrow \Gm$ is a surjection of reductive complex groups, it can be split up to an isogeny of $\Gm$: choosing such a splitting, we get a homomorphism $\phi: \Gm\rightarrow  (\Gm\times G)_x$ for which the composite $\Gm\xrightarrow{\phi} \Gm\times G \xrightarrow{\pi_{\Gm}} \Gm$ is $\Gm\xrightarrow{(\hspace{.5em})^m}\Gm$ for some $m>0$.  Taking quotients of  
$\on{Spec}(\mathbb{C})\xrightarrow{x} V\rightarrow \on{Spec}(\mathbb{C})$ yields a section of 
$(V/G)_{a_m}/\Gm = V/(\Gm\times G)  \rightarrow B\Gm$.
\end{proof}

\begin{lemma}\label{fixed-and-projection}
Assume a subgroup $C\subset G$ acts trivially on $V$; let $r: V/G\rightarrow V/(G/C)$ be the quotient. 
\begin{enumerate}
\item A $\mathbb{C}$-point $x\in V/G$ lies in $\fixed{V/G}$ if and only if $r(x)\in\fixed{V/(G/C)}$.
\item In particular, if $V/(G/C)$ is a scheme, then $x\in \fixed{V/G}$ if and only if $r(x)\in \big(V/(G/C)\big)^{\Gm}$.
\end{enumerate}
\end{lemma}
 \begin{remark}
Suppose $\mathcal{X} = V/G$ and that $X=V/(G/C)$ is a scheme.  Then by noetherianness of $X$ and Lemma \ref{fixed-and-projection}, there is an $m>0$ for which 
$\cX^{\Gm, a_m}(\mathbb{C})\rightarrow X^{\Gm}(\mathbb{C})$
 is already surjective.
\end{remark}

Consider now a reductive group $G$ acting on a smooth scheme $M$ with moment map $\mu: T^*M\rightarrow \mathfrak{g}^*$, and that $V=\mu\inv(0)$.  Let $\Gm$ act on $V\subset T^*M$ by fiberwise scaling, so $G\times\Gm$ acts on $V$.  We write $(x,\theta)\in T^*M$ to mean the point $\theta\in T^*_x(M)$.

\begin{lemma}\label{cot fixed points description}
Suppose $V = \mu\inv(0)\subseteq T^*M$ as above.  The $\mathbb{C}$-points of $\fixed{V/G}$ are the images of those $(x,\theta)\in V\subset T^*M$ such that there exist
\begin{enumerate}
\item a homomorphism $\lambda:\Gm\rightarrow \on{Stab}_G(x)$, and
\item an $m>0$ such that $\lambda(t)\cdot (x,\theta) = (x, t^m\theta)$ ($t\in\Gm$) in $V$.
\end{enumerate}
\end{lemma}

\subsection{A Localization Theorem for Product Actions}
As above, let $V$ be a reduced, separated scheme of finite type over $\mathbb{C}$, viewed as a topological space in the analytic topology, equipped with the action of a complex reductive algebraic group of the form $\Gm\times G$.  Fix a maximal compact subgroup $K$ of $G$, so that $S^1\times K\subset \Gm\times G$ is a maximal compact subgroup.  Consider the cohomology, always with $\mathbb{Q}$-coefficients,
$H^*\big((V/(\Gm\times G)\big) = H^*_{\Gm\times G}(V) = H^*_{S^1\times K}(V).$
We fix a generator $u\in H^2_{\Gm}(\on{pt})$, so $H^*_{\Gm}(\on{pt}) = \mathbb{Q}[u]$.  Write $\mathcal{Y} = V/(\Gm\times G)$.

\begin{prop}\label{localization theorem}   
The map $H^*\big(\mathcal{Y}\big)[u\inv] \rightarrow H^*\big(\mathcal{Y}_{\on{fixed}}\big)[u\inv]$ is an isomorphism.
\end{prop}
\begin{proof}
Write $\fixed{V/G} = F/G$ for the appropriate closed subset $F\subseteq V$.
The proof mimics that of Theorem 4.2 of \cite{Quillen}.  By the Mayer-Vietoris argument of \cite[Theorem~4.2]{Quillen}, it suffices to show that $H^*_{S^1\times K}(V\smallsetminus F)[u\inv]=0$.  Applying Remark 3.4 and Formula (3.1) of \cite{Quillen}, it suffices to show that the image of $u$ in $H^*_{S^1\times K}\big((S^1\times K)/H\big)$ is zero whenever $H\subset S^1\times K$ is a closed subgroup whose projection on $S^1$ is a proper (compact) subgroup of $S^1$.  

Thus, assume $H\subset S^1\times K$ is a closed subgroup whose image $\overline{H}$ in $S^1$ is finite.  Then the map $H^*_{S^1}(\on{pt})\rightarrow H^*_{S^1\times K}\big((S^1\times K)/H\big)$ factors as
\bd
H^*_{S^1}(\on{pt})\rightarrow H^*_{\overline{H}}(\on{pt}) \rightarrow H^*_{H}(\on{pt})\cong H^*_{S^1\times K}\big((S^1\times K)/H\big).
\ed
Since $\overline{H}$ is a finite group and we are working with $\mathbb{Q}$-coefficients, we have  $H^2_{\overline{H}}(\on{pt}) = 0$, and thus the image of $u$ in $H^*_{S^1\times K}\big((S^1\times K)/H\big)$ vanishes, as required.
\end{proof} 

We explain a similar statement for intersection cohomology (always with $\mathbb{Q}$-coefficients), which will suffice for our purposes.  As in the proof of Proposition \ref{localization theorem}, we restrict attention to the compact group $S^1\subset\Gm$.   We continue to write $\mathcal{X}=V/G$.  

We abusively write $IH^*_{S^1}(\widehat{\mathcal{X}}_{\on{fixed}}) = \underset{\stackrel{\longrightarrow}{U}}{\lim} \; IH^*_{S^1}(U)$, where $U$ ranges over $S^1$-invariant analytic open neighborhoods of $\mathcal{X}_{\on{fixed}}$ in $\mathcal{X}$.   
\begin{prop}\label{IH localization}
The natural map $IH^*(\mathcal{X}) \rightarrow IH^*_{S^1}(\widehat{\mathcal{X}}_{\on{fixed}})$ is an isomorphism.
\end{prop}
\noindent
The proof is a straightforward modification of the proof of Proposition \ref{localization theorem} as in \cite[Corollary~4.1.3]{Brylinski}.

\section{Interlude: Regular Nilpotent Elements}\label{sec:reg nilpotent}
Recall the surjective homomorphism $\pi_{\on{ad}}: G\rightarrow G_{\on{ad}}$ that factors the adjoint homomorphism
\bd
\on{Ad}: G\rightarrow GL(\mathfrak{g}), \hspace{1em} G_{\on{ad}} = \on{Ad}(G).
\ed
The group $G_{\on{ad}}$ is semisimple and $\on{ker}(\pi_{\on{ad}}) = Z$.  We have an isomorphism $[\mathfrak{g},\mathfrak{g}] \xrightarrow{\cong} \mathfrak{g}_{\on{ad}}$, yielding a canonical $G$-equivariant splitting
$\mathfrak{g} = \mathfrak{z}\oplus\mathfrak{g}_{\on{ad}}$ as in \eqref{g-splitting}.

Choose any regular nilpotent $e\in\mathfrak{g}_{\on{ad}}$.   It follows from the Jacobson-Morozov Theorem that there exists a 1-parameter subgroup $\gamma: \Gm\rightarrow T\subseteq G$ with the property that $\on{Ad}_{\gamma(t)}(e) = t^{m(\gamma)} e$ for some $m(\gamma)>0$, and that $\gamma$ can be chosen so that $m(\gamma)$ is either $1$ or $2$.  
\begin{convention}\label{gamma convention}
We fix a $\gamma$ for which $m(\gamma)$ is (whichever of $1$ and $2$ is) the minimum possible.
\end{convention}

The stabilizer $U_{\on{ad}}:=Z_{G_{\on{ad}}}(e)$ in $G_{\on{ad}}$ is a connected, commutative unipotent group \cite[Theorem~4.1]{Springer}; it follows that then
 $Z_G(e) = \pi_{\on{ad}}\inv\big(Z_{G_{\on{ad}}}(e)\big)$. 
\begin{lemma}\label{regular centralizer}
There is a unique connected unipotent subgroup $\mathsf{U}\subset G$ for which $Z_G(e) = Z\times \mathsf{U}$.  
\end{lemma}
\begin{proof}
Writing $\wt{G_{\on{ad}}}$ for the simply connected cover of $G_{\on{ad}}$, consider the commutative diagram
\bd
\xymatrix{
\wt{G_{\on{ad}}} \ar[d]^{p} \ar[dr]^{u_{\on{ad}}} & \\
G\ar[r]^{\pi_{\on{ad}}} & G_{\on{ad}}.
}
\ed
\bd
\text{Let} \hspace{1em} \wt{U}:= u_{\on{ad}}\inv(U_{\on{ad}}) \hspace{1em} \text{and} \hspace{1em} \mathsf{U}:= p(\wt{U})^\circ.
\ed
The homomorphism $\mathsf{U}\xrightarrow{\pi_{\on{ad}}} U_{\on{ad}}$ is surjective, since its image has finite index  by construction and $U_{\on{ad}}$ is connected; it has finite kernel since $u_{\on{ad}}$ has finite kernel.  Since a unipotent group in characteristic zero has no nontrivial finite subgroups, we get that $\pi_{\on{ad}}$ induces an isomorphism $\mathsf{U}\cong U_{\on{ad}}$.  Thus $\pi_{\on{ad}}\inv(U_{\on{ad}}) = Z\times\mathsf{U}$.  This proves existence.  
For uniqueness, note that if $\mathsf{U}, \mathsf{U}'$ are any two such subgroups, the composite $\mathsf{U}'\subset Z_G(e) = Z\times \mathsf{U}\twoheadrightarrow Z$ is trivial because $Z$ is a product of a torus and a finite group.  Thus $\mathsf{U}'\subset\mathsf{U}$.  Similarly $\mathsf{U}\subset \mathsf{U}'$, and we conclude $\mathsf{U}= \mathsf{U}'$.
\end{proof}
\begin{corollary}\label{equivalence of pairs}
\mbox{}
\begin{enumerate}
\item If $\on{Ad}_{\gamma_1(t)}(e) = t^me=\on{Ad}_{\gamma_2(t)}(e)$, then there exists $\eta: \Gm\rightarrow Z$ such that $\gamma_2 = \eta\cdot\gamma_1$.
\item For every 1-parameter subgroup $\lambda$ for which $\on{Ad}_{\lambda(t)}(e) =t^{m(\lambda)}(e)$, there exist $N>0$ and $\eta:\Gm\rightarrow Z$ such that $\lambda(t) = \eta\cdot\gamma(t)^N$.
\end{enumerate}
\end{corollary}
\begin{proof}
The product $\gamma_2\gamma_1\inv$ fixes $e$, hence is a 1-parameter subgroup in $Z\times\mathsf{U}$; but $\mathsf{U}$ contains none.

The second assertion follows from the first together with the fact that the set of $N\in\mathbb{Z}$ for which there exists a 1-parameter subgroup $\lambda$ satisfying $\on{Ad}_{\lambda(t)}(e) = t^Ne$, is a cyclic subgroup in $\mathbb{Z}$.
\end{proof}

\begin{lemma}\label{weights of 1-ps}
The $\on{Ad}_{\gamma(t)}$-action on $\mathfrak{g}_{\on{ad}}$ induces a splitting $\mathfrak{g}_{\on{ad}} = \mathfrak{n}^-\oplus\mathfrak{t}_{\on{ad}}\oplus \mathfrak{n}$, where $\mathfrak{t}_{\on{ad}}$ is a Cartan subalgebra of $\mathfrak{g}_{\on{ad}}$ and $\gamma(t)$ acts with  positive weights on $\mathfrak{n}$ and negative weights on $\mathfrak{n}^-$.
\end{lemma}
It follows from \cite[Theorem~2.2]{Humphreys} that, for any integer $n>0$, the centralizer $Z_{G_{\on{ad}}}(\gamma^n)$ 
satisfies  $Z_{G_{\on{ad}}}(\gamma^n)^\circ = \mathsf{T}_{\on{ad}}$ where $\mathsf{T}_{\on{ad}} = \pi_{\on{ad}}(\mathsf{T})$.  Then by Lemma \ref{regular centralizer} above and Theorem 2.2 of \cite{Humphreys} we get:
\begin{prop}\label{centralizer of pair}
For any $n>0$, we have $Z(e,\gamma^n) := Z_G(\gamma^n)\cap Z_G(e) = Z$.
\end{prop}

\section{Moduli of Semistable Higgs Bundles and Fixed Points}
We treat some properties of certain $\Gm$-fixed points on $\Higgst$.  A helpful discussion of the $\Gm$-action on the moduli space can be found in \cite{GRR}.

\subsection{Analysis of Fixed Points on $\Higgst$}
Suppose $(P_1, \theta_1, \lambda_1)$, $(P_2,\theta_2,\lambda_2)$  are points of $\Higgst^{\hGm}_{\on{reg}}$ (so $\theta_1, \theta_2$ are regular nilpotent in every fiber): that is, $P_i$ is a principal $G$-bundle, $\theta_i$ is a Higgs field on $P_i$, and $\lambda_i$ is a 1-parameter subgroup of $\on{Aut}_G(P_i)$ that rescales $\theta_i$, i.e., $\on{Ad}_{\lambda_i(t)}(\theta_i) = t^{m_i}\theta_i$ for some integer $m_i$.
  Replacing $\lambda_1, \lambda_2$ by appropriate powers, we assume that $\theta_i$ is an eigenvector for $\lambda_i$ with the same exponent: $\on{Ad}_{\lambda_i(t)}(\theta_i) = t^N\theta_i$.   
\begin{lemma}\label{fixed points under Bun(Z,C)}
There exist a principal $Z$-bundle $P_Z$, a 1-parameter subgroup $\eta: \Gm\rightarrow Z$ and an isomorphism $\phi: P_Z\times_Z P_1 \xrightarrow{\cong} P_2$ of $G$-bundles so that 
$\phi(\theta_1) = \theta_2$ and $\phi(\eta\lambda_1) = \lambda_2$.
\end{lemma}
\begin{proof}
Choose a point $x\in C$.  By Corollary \ref{equivalence of pairs}, there exist a 1-parameter subgroup $\eta:\Gm\rightarrow Z$ and an isomorphism $(P_1)_x\cong (P_2)_x$ of principal homogeneous $G$-spaces that takes $(\theta_1)_x$ to $(\theta_2)_x$ and identifies $\eta\lambda_1$ with $\lambda_2$ on the fiber over $x$.  Now, replacing $\lambda_1$ by $\eta\lambda_1$, we may assume that the fibers over $x$ of $(P_1,\theta_1, \lambda_1)$ and $(P_2,\theta_2,\lambda_2)$ are isomorphic.  Since $\on{Hom}_{\on{gp}}(\Gm,Z)$ is discrete, we conclude via Proposition \ref{centralizer of pair} that the isomorphism bundle $P_Z:=\underline{\on{Isom}}\big((P_1, \theta_1, \lambda_1), (P_2,\theta_2,\lambda_2)\big)$ is a $Z$-torsor.   This completes the proof.
\end{proof}

\subsection{Stable $G$-Higgs Bundles with Regular Nilpotent Higgs Field}
\begin{lemma}
Suppose that $(P,\theta)$ is a stable $G$-Higgs bundle with regular nilpotent Higgs field.  Then $(P,\theta)$ is simple, i.e.,  $\on{Aut}(P,\theta) = Z$.
\end{lemma}
\begin{proof}
By Lemma \ref{regular centralizer}, the group scheme $\underline{\on{Aut}}(P,\theta)$ of fiberwise automorphisms of $(P,\theta)$ is the product of the constant group scheme $Z$ over $C$ and a unipotent group scheme $\mathsf{U}(P,\theta)$ over $C$.  
Because $(P,\theta)$ is stable, Proposition 3.14 of \cite{GO} implies that every global section of the automorphism group scheme, i.e. element of $\on{Aut}(P,\theta)$, is semisimple.  The conclusion follows.
\end{proof}

\begin{prop}\label{action map}
Suppose that $(P,\theta)$ is any stable $G$-Higgs bundle with regular nilpotent Higgs field.  Consider the morphism $a = a_{(P,\theta)}:\Bun(Z,C)\rightarrow \mathcal{M}_{\on{Higgs}}^{\on{st-spl}}(G,C)$ obtained by restricting the action morphism $\Bun(Z,C)\times \Higgst \rightarrow \Higgst$ to $\Bun(Z,C)\times \{(P,\theta)\}$.  Then we obtain a commutative diagram
\bd 
\hspace{5em}
\xymatrix{
\Bun(Z,C)\ar[r]^{a} \ar[d]^{p_2} & \mathcal{M}_{\on{Higgs}}^{\on{st-spl}}(G,C) \ar[r]^{i}\ar[d]^{p_1} & \Higgst\ar[d]^{p} & \\
\underline{\Bun}(Z,C)\ar[r]^{\overline{a}} & M_{\on{Higgs}}^{\on{st-spl}}(G,C) \ar[r]^{\overline{i}} & M_{\on{Higgs}}^{\on{ss}}(G,C), & \text{where}
}
\ed
\begin{enumerate}
\item the morphisms $p_2, p_1$ are $BZ$-torsors, realizing $\Bun(Z,C)$ as the fiber product along $\overline{a}$ and $p_1$, and thus also realizing $\Bun(Z,C)$ as the fiber product along $\overline{i}\circ\overline{a}$ and $p$;
\item the morphisms $\overline{i}, i$ are open immersions; and
\item  the morphisms $\overline{a}$, $\overline{i}\circ\overline{a}$, $a$,  and $i\circ a$ are closed immersions.
\end{enumerate}
\end{prop}
\begin{proof}
The existence of the commutative diagram is immediate from the constructions.  Assertion (1) is immediate from the facts that $p_2$ and $p_1$ are both $BZ$-torsors and that $a$ is $BZ$-equivariant.  Assertion (2) is immediate from Proposition \ref{Higgs basics}(3).

To prove assertion (3), we proceed as follows.  By Proposition \ref{properness of BunZ}, each component of $\underline{\Bun}(Z,C)$ is a projective variety, hence the restrictions of $\overline{a}$ and $\overline{i}\circ\overline{a}$ to each component of $\underline{\Bun}(Z,C)$ are proper.  It now follows from Proposition \ref{how components behave} that $\overline{a}$ and $\overline{i}\circ\overline{a}$ are of finite type, hence $\overline{a}$ and $\overline{i}\circ\overline{a}$ are proper (the targets are separated, so properness lifts).  

If we can show that $\overline{a}$ and $\overline{i}\circ\overline{a}$ are bijections onto their images on the level of $\mathbb{C}$-points, it will then follow from Corollary \ref{injectivity} that they separate tangents, hence are closed immersions.  
Thus, suppose that $(P_1,\theta_1)$ and $(P_2,\theta_2)$ are in the image of $\overline{a}$.  As in the proof of Lemma \ref{fixed points under Bun(Z,C)}, the automorphism group scheme $\Aut(P_2,\theta_2)$ is the product of the constant group scheme $Z$ over $C$ and a unipotent group scheme $\mathsf{U}(P,\theta)$ over $C$; similarly, 
it follows from Lemma \ref{regular centralizer} that the isomorphism bundle $\underline{\on{Isom}}\big((P_1,\theta_1),(P_2,\theta_2)\big)$ is a 
$\Aut(P_2,\theta_2)$-torsor, which evidently is trivial if and only if $(P_1,\theta_1)\cong(P_2,\theta_2)$.  We assume, without loss of generality in the proof, that $(P_2,\theta_2) = (P,\theta)$.  

We thus consider the long exact sequence associated to 
$1\rightarrow Z\times C \rightarrow \Aut(P,\theta) \rightarrow \mathsf{U}(P,\theta)\rightarrow 1$,
\bd
1\rightarrow H^0(C,Z) \xrightarrow{F} H^0(C,\Aut(P,\theta)) \rightarrow H^0(\mathsf{U}(P,\theta)) \rightarrow H^1(C,Z)\xrightarrow{\overline{a}} H^1(C,\Aut(P,\theta)) \rightarrow \dots.
\ed
As noted above, $(P,\theta)$ is simple, and it follows that $F$ is an isomorphism.  It follows from the ``injectivity of tangent spaces'' assertion of Corollary \ref{injectivity} that the fiber of $\overline{a}$ over $(P,\theta)$ is finite, hence the image of the injective group homomorphism $H^0(\mathsf{U}(P,\theta)) \rightarrow H^1(C,Z)$ is finite, and in particular $H^0(\mathsf{U}(P,\theta))$ is finite.  But the fibers of the group scheme $\mathsf{U}(P,\theta)$ are connected unipotent groups, which have no nontrivial finite subgroups, implying that $H^0(\mathsf{U}(P,\theta))=\{e\}$, and thus $\overline{a}$ is injective.  Thus $\overline{a}$ and $\overline{i}\circ\overline{a}$ are indeed bijections onto their images, implying that they are closed immersions.  

It now follows from assertion (1) and the conclusion of the previous paragraph that $a$ and $i\circ a$ are also closed immersions.  
\end{proof}

\section{The Main Construction and the Proof of Theorem \ref{main thm}}
\subsection{The Main Construction}
As in Section \ref{sec:reg nilpotent}, assume that we have fixed (Convention \ref{gamma convention}) a 1-parameter subgroup $\gamma:\Gm\rightarrow G$ and a regular nilpotent $e\in \mathfrak{g}$ that is an eigenvector of $\on{Ad}_{\gamma(t)}$ with eigenvalue $t^{m(\gamma)}$, where $m(\gamma)$ is either $1$ or $2$.

\begin{defn}\label{specific bundle def}
We define 
\bd
P_0:= \mathsf{K}^{-1/2}\times_{\gamma_1} G, \; \text{where} \; \gamma_1(t) = \begin{cases} \gamma(t) & \text{if $m(\gamma) = 2$},\\
\gamma(t^2) & \text{if $m(\gamma)=1$.}
\end{cases}
\ed
We let $\eta_0$ be the connected component of $\Bun(G,C)$ containing $P_0$.  
Thus $P_0\times_G\mathfrak{g} = \mathsf{K}^{-1/2}\times_{\gamma_1} \mathfrak{g}$.
\end{defn}
\begin{lemma}
\mbox{}
\begin{enumerate}
\item $\gamma$ determines a $1$-parameter group $\lambda$ of automorphisms of $P_0$ by 
$\lambda(t)\cdot (a,p) = (a,\gamma_1(t)p)$ for $a\in \spin$, $p\in G$. 
\item 
The element $e\in \mathfrak{g}$ determines a section $\theta_0$ of 
$(P_0\times_G\mathfrak{g})\otimes\Omega^1_C$.  Moreover, $\lambda(t)^*\theta_0 = t^2\theta_0$.
\end{enumerate}
\end{lemma}
\begin{proof}
Part (1) is clear by construction of $P_0$.  For part (2), note that 
\bd
(P_0\times_G \mathfrak{g})\otimes \Omega^1_C \cong \mathsf{K}^{-1/2} \times_{\Gm}\big(\mathfrak{g}\otimes \mathbb{A}^1(-2)\big),
\ed
where $\Gm$ acts on $\mathfrak{g}$ via $\gamma_1$.  Since $e\in\mathfrak{g}$ has weight $2$ with respect to $\gamma_1$, it determines an invariant element of $\mathfrak{g}\otimes\mathbb{A}^1(-2)$.
Thus, $e$ descends to a section $\theta_0$ as asserted.  That $\theta_0$ has weight $2$ with respect to $\lambda$ is clear from the construction.
\end{proof}

\noindent
The Higgs pair $(P_0,\theta_0)$ is the image of the Hitchin section at $0$ in the Hitchin base \cite[Section~5]{Hitchin2}.
\begin{lemma}\label{stability of P_0}
The Higgs pair $(P_0,\theta_0)$ is stable.
\end{lemma}
\begin{proof}
By the Jacobson-Morozov Theorem,  $e$ lies in a unique Borel subalgebra $\mathfrak{b}\subset\mathfrak{g}$; thus
any parabolic reduction of $P_0$ that is compatible with $\theta_0$  is of the form $\mathsf{K}^{-1/2}\times_{\gamma_1}P$ for a parabolic $B\subseteq P\subseteq G$.  Now pulling back the relative tangent sheaf of $P_0\times_G (\mathfrak{g}/\mathfrak{p})\rightarrow C$ along the section that determines the $P$-reduction, we get a bundle of the form $\mathsf{K}^{-1/2}\times_{\gamma_1} \mathfrak{g}/\mathfrak{p}$ for a parabolic $\mathfrak{b}\subseteq\mathfrak{p}\subseteq\mathfrak{g}$.  It now follows from  Lemma \ref{weights of 1-ps} that this bundle has negative degree. 
\end{proof}

Applying Lemma \ref{fixed points under Bun(Z,C)}, we conclude:
\begin{corollary}\label{crucial corollary}
The pair $(P_0,\theta_0)$ defines a commutative diagram as in Proposition \ref{action map}.  The image of $a$ is exactly
\bd
\Higgst_{\on{fixed}}\cap \Higgst_{\on{reg}}.
\ed
\end{corollary}

\subsection{Proof of Theorem \ref{main thm}}
That $H^*\big(\Bun(Z,C)\big)$ acts trivially on $H^*\big(\Bun(G,C)\big)$ is clear from the complex-analytic construction of $\Bun(G,C)$ which realizes the stack as a quotient of an (infinite-dimensional) affine space by a Lie group (again infinite-dimensional): the space $\underline{\Bun}(Z,C)$ acts on the affine space of $\overline{\partial}$-operators, hence the action is homotopically trivial.  Thus, assertions (3) and (4) of the theorem are immediate from assertion (1).

From now on we fix the pair $(P_0,\theta_0)$ of the Main Construction.  By Corollary \ref{crucial corollary}, the image of the corresponding morphism $\Bun(Z,C)\rightarrow \Higgst$ is $\Higgst_{\on{fixed}}\cap\Higgst_{\on{reg}}$.  The latter set is both open in 
$\Higgst_{\on{fixed}}$---because it is the intersection with the open subset of regular $G$-Higgs pairs---and closed, because it is the image of a proper morphism.  Thus, the image of $\Bun(Z,C)\rightarrow \Higgst$ consists of a union of connected components of $\Higgst_{\on{fixed}}$.  The morphism $\Bun(Z,C)\rightarrow\Higgst$ being a closed immersion by Proposition \ref{action map}, we conclude:
\begin{fact}\label{homeomorphism fact}
The morphism $\Bun(Z,C)\rightarrow\Higgst$ determined by $(P_0,\theta_0)$  is a homeomorphism from $\Bun(Z,C)$ onto a union of connected components of $\Higgst_{\on{fixed}}$.  
\end{fact}

The finite abelian group stack $\Bun(\pi_0(Z),C)$---which is nontrivial by our main assumption that $\pi_0(Z)$ is nontrivial---acts on $\Higgst$, inducing compatible actions on $H^*\big(\Higgst\big)$ and $H^*(\Higgst_{\on{fixed}})$.  It follows from Proposition \ref{action on connected components} that $\Bun(\pi_0(Z),C)$ preserves connected components of $\Bun(G,C)$, hence by \cite{GO} preserves connected components of $\Higgst$.

Consider the Leray spectral sequence for $\Gm$-equivariant cohomology,
\bd
E_2^{\ast,\ast} = H^*(\MHiggs,\mathbb{Q})\otimes H^*_{\Gm}(\on{pt},\mathbb{Q}) \implies H^*_{\Gm}(\MHiggs,\mathbb{Q}).
\ed
Writing 
$H^*_{\Gm}(\on{pt},\mathbb{Q}) \cong \mathbb{Q}[u]$ and localizing to $\mathbb{Q}[u^{\pm 1}]$,  the spectral sequence becomes
\bd
E_2^{\ast,\ast} = H^*(\MHiggs,\mathbb{Q})\otimes H^*_{\Gm}(\on{pt},\mathbb{Q})[u^{\pm 1}] \implies H^*_{\Gm}\big(\MHiggs,\mathbb{Q}\big)[u^{\pm 1}].
\ed
\begin{remark}\label{MHS remark 1}
We note that the mixed Hodge structure on $H^*_{\Gm}(\on{pt},\mathbb{Q})[u^{\pm 1}]$ is pure; thus only the pure part of $H^*(\MHiggs,\mathbb{Q})$ contributes to the pure part of 
$H^*_{\Gm}\big(\MHiggs,\mathbb{Q}\big)[u^{\pm 1}]$.
\end{remark}
Applying Proposition \ref{localization theorem} and Proposition \ref{fixed points and localization}, the spectral sequence abuts to $H^*_{\Gm}(\Higgst_{\on{fixed}},\mathbb{Q})[u\inv]$.  
The spectral sequence is equivariant with respect to all $\Gm$-compatible automorphisms of $\MHiggs$: in particular, with respect to the  $\Bun(\pi_0(Z),C)$-action.  
\begin{lemma}\label{reg rep claim}
 The graded representation $H^*_{\Gm}(\Higgst_{\on{fixed}},\mathbb{Q})[u\inv]$ of $\Bun(\pi_0(Z),C)$ contains a copy of $\mathbb{Q}[\Bun(\pi_0(Z),C)]$ lying in graded degree $0$ (and consisting of pure classes for the mixed Hodge structure).
\end{lemma}
\begin{proof}[Proof of lemma]
The group $\Bun(\pi_0(Z),C)$ acts freely on $\pi_0\Bun(Z,C)$; hence $H^*(\Bun(Z,C),\mathbb{Q})$ contains a copy of the regular representation of $\Bun(\pi_0(Z),C)$ in degree $0$, implying that 
\bd
H^*_{\Gm}(\Bun(Z,C),\mathbb{Q}) = H^*(\Bun(Z,C),\mathbb{Q})\otimes H^*_{\Gm}(\on{pt})
\ed does as well (here $\Gm$ acts trivially on $\Bun(Z,C)$);
hence by Fact \ref{homeomorphism fact}, the same assertion is true of $H^*_{\Gm}(\Higgst_{\on{fixed}},\mathbb{Q})[u\inv]$.  
\end{proof}
We continue with the proof of the theorem.  Since $H^*_{\Gm}(\Higgst_{\on{fixed}},\mathbb{Q})[u\inv]$ can be obtained by taking iterated $\Bun(\pi_0(Z),C)$-equivariant subquotients of  $H^*(\Higgst,\mathbb{Q})\otimes H^*_{\Gm}(\on{pt},\mathbb{Q})[u^{\pm 1}]$, it then follows that the latter contains a copy of $\mathbb{Q}[\Bun(\pi_0(Z),C)]$ lying in graded degree $0$.  
Finally, note that for any graded vector space $V^*$, the natural map 
$V^* \rightarrow V^*[u^{\pm 1}]_0$ acting by multiplication by $u^{-k}$ on $V^k$, is an isomorphism. 
It then follows from this observation and Claim \ref{reg rep claim} that  $H^*(\Higgst,\mathbb{Q})$ contains a copy of 
$\mathbb{Q}[\Bun(\pi_0(Z),C)]$, as desired.

It remains to prove a similar assertion for the moduli space $\Higgsp$.  
It follows from Proposition \ref{action map}, Proposition \ref{Higgs basics}(3), Lemma \ref{fixed-and-projection}(2), and Fact \ref{homeomorphism fact} that $\underline{\Bun}(Z,C)\rightarrow \Higgsp^{\Gm}$ is a homeomorphism onto a union of connected components of $\Higgsp^{\Gm}$.
  Since $\pi_0\left(\Bun(Z,C)\right) = \pi_0\left(\underline{\Bun}(Z,C)\right)$, we conclude
  that $\Bun(\pi_0(Z),C)$ again acts nontrivially on $H^*(\Higgsp,\mathbb{Q})$.  Further, it follows
  from the functoriality of the Leray spectral sequence that the homomorphism 
$H^*(\Higgsp,\mathbb{Q})\rightarrow H^*(\Higgst,\mathbb{Q})$ must take some nontrivial $\Bun(\pi_0(Z),C)$-isotypic components of the domain (corresponding to part of $H^0(\Higgsp^{\Gm})$) isomorphically onto their images.  This completes the proof of assertion (1).

The intersection cohomology assertions (2) follow from the same proof applied to $IH^*$ rather than $H^*$: the crucial observation is that the union of connected components given by the image of $\Bun(Z,C)$ lies in the rationally smooth locus of $\Higgst$, where the colimit appearing in the target of Proposition \ref{IH localization} reduces to the singular cohomology of the (smooth) fixed stack $(\Higgst)_{\on{fixed}}$.  This completes the proof.
\hfill\qedsymbol
\begin{remark}\label{MHS remark 2}
In light of Remark \ref{MHS remark 1}, the proof shows that the representation $\mathbb{Q}[\Bun(\pi_0(Z),C)]$ appears in the pure part of the mixed Hodge structures on $H^*(\Higgst,\mathbb{Q})$ and $H^*(\Higgsp,\mathbb{Q})$.
\end{remark}

\end{document}